\documentclass[12pt]{article}

\usepackage[T1]{fontenc}
\usepackage{avant}
\usepackage[cp1250]{inputenc}
\usepackage{amsmath,amssymb,amsthm}
\usepackage{caption}
\usepackage{geometry}
\usepackage{float}

\usepackage{makeidx}
\usepackage[matrix,arrow,curve]{xy}
\usepackage{boxedminipage}
\usepackage{epic}
\usepackage{longtable}
\usepackage{ifthen}
\usepackage{tabularx}
\usepackage{parskip}
\usepackage{verbatim}

\title{ Standard compact Clifford-Klein forms and Lie algebra decompositions}
\author{Maciej Boche\'nski and Aleksy Tralle}

\begin{document}

\newtheorem{theorem}{Theorem}
\newtheorem{proposition}{Proposition}
\newtheorem{lemma}{Lemma}
\newtheorem{definition}{Definition}
\newtheorem{example}{Example}
\newtheorem{note}{Note}
\newtheorem{corollary}{Corollary}
\newtheorem{remark}{Remark}
\newtheorem{question}{Question}

\maketitle{}
\abstract{We find relations between real root decompositions of Lie triples $(\mathfrak{g},\mathfrak{h},\mathfrak{l})$ corresponding to compact standard Clifford-Klein forms, under the assumption that $(\mathfrak{g},\mathfrak{h},\mathfrak{l})$ is not a Lie algebra decomposition in the sense of Onishchik. This enables us to find new classes of homogeneous spaces $G/H$ of simple real Lie groups which do not admit standard compact Clifford-Klein forms. In particular, we show that proper R-regular subalgebras $\mathfrak{h}$ of simple real Lie algebras $\mathfrak{g}$ never generate homogeneous spaces $G/H$ which admit compact Clifford-Klein forms. }

\section{Introduction}
\subsection{The problem} Let $G$ be a connected real semisimple linear Lie group and $H\subset G$ a closed connected reductive subgroup.
\begin{definition}[T. Kobayashi, \cite{k-k},\cite{kob}]\label{def:standard}
{\rm A homogeneous space $G/H$ admits a {\it standard Clifford-Klein form}, if there exists a closed reductive subgroup $L\subset G$ such that $L$ acts properly and co-compactly on $G/H.$ }
\end{definition}
\noindent This notion was first introduced in \cite{kob}. A list of such spaces was given in \cite{kob1} and \cite{kob-y}. The term "standard" was first used in \cite{k-k}.
 Standard Clifford-Klein forms are determined by triples $(G,H,L)$, where $G$ is a semisimple real linear  Lie group, and $H$ and $L$ are reductive subgroups. The general theory \cite{kob} shows that we can restrict ourselves to considering triples $(\mathfrak{g},\mathfrak{h},\mathfrak{l})$ of semisimple Lie algebras, and we will do that throughout this article.
\begin{theorem}\label{thm:table}
Assume that $\mathfrak{g}$ is absolutely simple, and $\mathfrak{h}$ and $\mathfrak{l}$ are reductive. All triples $(\mathfrak{g},\mathfrak{h},\mathfrak{l})$  listed in Table 1 yield standard compact Clifford-Klein forms. 
\end{theorem}

\begin{center}
 \begin{table}[h]
 \centering
 {\footnotesize
 \begin{tabular}{| c | c | c |}
   \hline
   \multicolumn{3}{|c|}{ \textbf{\textit{Standard Clifford-Klein forms}}} \\
   \hline                        
   $\mathfrak{g}$ & $\mathfrak{h}$ & $\mathfrak{l}$ \\
   \hline
		$\mathfrak{su}(2,2n)$ & $\mathfrak{sp}(1,n)$ & $\mathfrak{su}(1,2n)$ \\
		\hline
		$\mathfrak{su}(2,2n)$ & $\mathfrak{sp}(1,n)$ & $\mathfrak{u}(1,2n)$ \\
		\hline
		$\mathfrak{so}(2,2n)$ & $\mathfrak{so}(1,2n)$ & $\mathfrak{su}(1,n)$ \\
		\hline
		$\mathfrak{so}(2,2n)$ & $\mathfrak{so}(1,2n)$ & $\mathfrak{u}(1,n)$ \\
		\hline
		$\mathfrak{so}(4,4n)$ & $\mathfrak{so}(3,4n)$ & $\mathfrak{sp}(1,n)$ \\
		\hline
		$\mathfrak{so}(4,4n)$ & $\mathfrak{so}(3,4n)$ & $\mathfrak{sp}(1,n)\times\mathfrak{sp}(1)$ \\
		\hline
		$\mathfrak{so}(4,4n)$ & $\mathfrak{so}(3,4n)$ & $\mathfrak{sp}(1,n)\times\mathfrak{so}(2)$ \\
		\hline
		$\mathfrak{so}(4,4n)$ & $\mathfrak{so}(3,4n)$ & $\mathfrak{sp}(1,n)$ \\
		\hline
		$\mathfrak{so}(3,4)$ & $\mathfrak{g}_{2(2)}$ & $\mathfrak{so}(1,4)\times\mathfrak{so}(2)$ \\
		\hline
		$\mathfrak{so}(3,4)$ & $\mathfrak{g}_{2(2)}$ & $\mathfrak{so}(1,4)$ \\
		\hline
		$\mathfrak{so}(8,8)$ & $\mathfrak{so}(7,8)$ & $\mathfrak{so}(1,8)$ \\
		\hline
		$\mathfrak{so}(4,4)$ & $\mathfrak{so}(3,4)$ & $\mathfrak{so}(1,4)\times\mathfrak{so}(3)$ \\
		\hline
		$\mathfrak{so}(4,4)$ & $\mathfrak{so}(3,4)$ & $\mathfrak{so}(1,4)\times\mathfrak{so}(2)$ \\
		\hline
		$\mathfrak{so}(4,4)$ & $\mathfrak{so}(3,4)$ & $\mathfrak{so}(1,4)$ \\
		\hline
 \end{tabular}
 }
 \caption{
 }
 \label{tttab1}
 \end{table}
\end{center}
The proof of this result is well known: all triples in Table 1 are Lie algebra decompositions in the sense of Onishchik: $\mathfrak{g}=\mathfrak{h}+\mathfrak{l}$ \cite{on}. It follows from this work, that $G=H\cdot L$ as well, and that $L$ acts transitively on $G/H$. One obtains a natural diffeomorphism $G/H\cong L/L\cap H$, and since $H\cap L$ is compact, $L$ acts properly and co-compactly on $G/H$.  
It is important to note that these are {\it the only} known examples of compact standard Clifford-Klein forms (with $\mathfrak{g}$ simple and $\mathfrak{h}$ proper and non-compact). Hence, the main aim of this article is to answer the following  question.
\begin{question}\label{q:non-on} Are there Lie algebra triples $(\mathfrak{g},\mathfrak{h},\mathfrak{l})$ which can yield compact Clifford-Klein forms and which are not  Lie algebra decompositions?
\end{question}
\noindent Our main results establish some properties of the root space decompositions of such hypothetical pairs. In particular, we show that none of the R-regular subalgebras $\mathfrak{h}$ can yield a  non-Onishchik triple $(\mathfrak{g},\mathfrak{h},\mathfrak{l})$ which generates a standard compact Clifford-Klein form. It is well known that all semisimple subalgebras of real simple Lie algebras fall into two classes: R-regular, and S-regular (also in the real case, see \cite{FG}). Thus, we give a complete  answer to Question \ref{q:non-on} for one of these two classes of semisimple subalgebras.    

There are various reasons for studying standard compact Clifford-Klein forms. One of them comes from  spectral analysis on pseudo-Riemannian manifolds \cite{k-k}, \cite{k-k1}. In this area, one is interested in constructing  quotients $M_{\Gamma}=\Gamma \backslash M$ of discrete groups of isometries of a pseudo-Riemannian homogeneous  manifold $M=G/H$. However, it may happen that $M$ does not admit infinite discontinuous groups of isometries (the Calabi-Markus phenomenon).  On the other hand, there exist examples of manifolds $G/H$ admitting large discontinuous groups such that $M_{\Gamma}$ is compact, or of finite volume. In this case, we will call $M_{\Gamma}$ a {\it compact Clifford-Klein form}, or a Clifford-Klein form of finite volume. Thus, it is a challenging problem to understand when do such compact quotients exist. Suppose a Lie subgroup $L\subset G$ acts properly on $G/H$. Then the action of any discrete subgroup $\Gamma$ of $L$  is automatically properly discontinuous, and this action is free whenever $\Gamma$ is torsion-free. Moreover, if $\Gamma$ is co-compact in $L$, and $L$ acts co-compactly on $G/H$, then the quotient $\Gamma\backslash G/H$ is compact. A quotient $\Gamma\backslash G/H$ is called standard if $\Gamma$ is contained in a reductive subgroup $L$ acting properly on $G/H$. Thus, the first question which arises is to find standard quotients, which is equivalent to finding triples $(G,H,L)$ as in Definition \ref{def:standard}. Partial results in this direction were obtained in \cite{bjt} and \cite{tojo}. 

One of the most challenging problems in this area of research is a problem of T. Kobayashi: {\it does  any homogeneous space $G/H$ admitting compact Clifford-Klein form admits also a standard one?} Clearly, one of the ways to attack it is to understand all standard Clifford-Klein forms.

In a more geometric vein, if $\Gamma\subset G$ acts freely and co-compactly on a  homogeneous space $G/H$ with non-compact $G$ and $H$, any $G$-invariant geometric structure descends onto the quotient, which yields compact   manifolds with such structures. The topology of such manifolds can be understood, and therefore, these manifolds can serve as testing examples. For instance, we mention \cite{bt}.

In the context of this work, one should mention that there are many non-existence results for compact Clifford-Klein forms. For example, there are no standard Clifford-Klein forms of exceptional Lie groups \cite{bjt}. Topological obstructions to compact Clifford-Klein forms were found in \cite{Mo},\cite{M1}, \cite{T} together with examples. A procedure of constructing homogeneous spaces with no compact Clifford-Klein forms was obtained in \cite{bjst}. On the positive side, all compact standard Clifford-Klein of symmetric spaces were classified in \cite{tojo}.

\subsection{Results}
In order to formulate our results we need some notation. The details are described separately in Section \ref{sec:prelim}.
\vskip6pt
Assume that $(\mathfrak{g},\mathfrak{h},\mathfrak{l})$ is a standard triple. Consider the complexifications $\mathfrak{g}^c,\mathfrak{h}^c,\mathfrak{l}^c$, and  the compatible Cartan decompositions 
$$\mathfrak{g}=\mathfrak{k}+\mathfrak{p},\mathfrak{h}=\mathfrak{k}_h+\mathfrak{p}_h+\mathfrak{p}_h,\mathfrak{l}=\mathfrak{k}_l+\mathfrak{p}_l.$$
Denote by $\theta$ the corresponding Cartan involution and  by $\tau$ the complex conjugation of $\mathfrak{g}^c$ with respect to the compact real form $\mathfrak{g}_u=\mathfrak{k}+i\mathfrak{p}$. Let $\mathfrak{a}\subset\mathfrak{p}$ be a maximal abelian subspace and $\mathfrak{m}_0$ be a centralizer of $\mathfrak{a}$ in $\mathfrak{k}$. Choose a maximal Cartan subalgebra $\mathfrak{t}$ in $\mathfrak{m}_0$. Let $\Sigma\subset\mathfrak{a}^*$ be the real root system of $\mathfrak{g}$ determined by $\mathfrak{a} .$ The subalgebra $\mathfrak{j}^c=(\mathfrak{t}+\mathfrak{a})^c$ is a Cartan subalgebra of $\mathfrak{g}^{c}$ so we can take the corresponding root system $\Delta$ of $\mathfrak{g}^c$. Without loss of generality we assume that $\mathfrak{a}_h,\mathfrak{a}_l\subset\mathfrak{a}$ (denote by $\Sigma_h$, $\Sigma_l$ be the real root systems of $\mathfrak{h}$ and $\mathfrak{l}$, respectively).  Let 
$$\mathfrak{g}=\mathfrak{k}+\mathfrak{a}+\mathfrak{n},\mathfrak{h}=\mathfrak{k}_h+\mathfrak{a}_h+\mathfrak{n}_h,\mathfrak{l}=\mathfrak{k}_l+\mathfrak{a}_l+\mathfrak{n}_l$$
be the compatible Iwasawa decompositions. Denote by $B$ the Killing form of $\mathfrak{g} .$ For any $\gamma\in\Sigma$ we denote by $\mathfrak{g}_{\gamma}$ the corresponding root space and by $\mathfrak{g}_{\gamma}^c$ its complexification.
\begin{proposition}\label{prop:iwasawa} Assume that $(G',H',L')$ is a standard triple. Then there is an equivalent standard triple $(G,H,L)$ such that 
$$\mathfrak{n}=\mathfrak{n}_h\oplus\mathfrak{n}_l.$$
\end{proposition}
\noindent This Proposition enables us to get relations between real root spaces of $\mathfrak{g},\mathfrak{h}$ and $\mathfrak{l}$. The proof is contained in Subsection \ref{subsec:n=n1+n2}.

Assume that $0\not=X\in\mathfrak{t}$. Define
$$\Delta_m=\{\alpha_c\in\Delta\,|\,\alpha_c|_{\mathfrak{a}}=0\}\setminus\{0\},$$
$$\Delta_m^{+}=\Delta^+\setminus\Delta_m,\Delta_m^{-}=\Delta^{-}\setminus\Delta_m,$$
$$\Delta_p=\{\alpha_c\in\Delta_m^{+}\,|\,\alpha_c(iX)>0\},\,\Delta_n=\{\alpha_c(iX)<0\},$$
$$\Delta_0=\{\alpha_c\in\Delta^+\,|\,\alpha_c(iX)=0\}.$$
\begin{theorem}\label{thm:root-decomp} Let
$$Z=\sum_{\alpha_c\in\Delta_p\cup\Delta_n}\mathfrak{g}_{\alpha_c}\subset\mathfrak{n}^c.$$
Let $\pi:\mathfrak{n}^c=\sum_{\alpha_c\in\Delta_m^+}\mathfrak{g}_{\alpha_c}\rightarrow Z$ be the natural projection, $Z_h=\pi(\mathfrak{n}_h^c), Z_l=\pi(\mathfrak{n}_l^c)$. Assume that $(\mathfrak{g},\mathfrak{h},\mathfrak{l})$ is a standard triple and that there exists a non-zero $X\in\mathfrak{t}$ such that $X\not\in (\mathfrak{h}+\mathfrak{l})$. We may assume that $X$ is $B$-orthogonal to $\mathfrak{h}+\mathfrak{l} .$ Then:
\begin{enumerate}
\item There exists a basis of $Z_h$ of the form
$$S_h^i=x_{\alpha_i}+\sum_{l=1}^ka^i_{k+l}x_{\alpha_{k+l}},a^i_{k+l}\in\mathbb{C},\alpha_i\in\Delta_p,\alpha_{k+l}\in\Delta_n .$$
\item For any $S_h^i\in Z_h$ $\alpha_i|_{\mathfrak{a}_h}\in\Sigma_h$.
\item Each complexified real root space $\mathfrak{h}_{\gamma}^c$ is spanned by vectors of the form
$$S_h^{i_1}+Q_1,...,S_h^{i_s}+Q_s,Q_{s+1},...,Q_{s+w},s+w=\dim\,{h}_{\gamma}^c,$$
where $\alpha_{i_1},...,\alpha_{i_s}$ are all roots from $\Delta_p$ whose restrictions onto $\mathfrak{a}_h$ coincide with $\gamma$, while all $Q_j$ satisfy the conditions
$$Q_j\in\sum_{\alpha_{c}\in\Delta_{0}, \ \alpha_{c}|_{\mathfrak{a}_{h}}=\gamma}\mathfrak{g}_{\alpha_{c}}.$$
\end{enumerate}
Analogous conditions hold for $\mathfrak{l}_{\gamma}^c$.
\end{theorem}
\noindent This description of the complexified root spaces $\mathfrak{h}_{\gamma}^c$ and $\mathfrak{l}_{\gamma}^c$ yield new classes of homogeneous spaces $G/H$ which do not admit standard compact Clifford-Klein forms. We say that $\mathfrak{h}$ is a regular subalgebra of $\mathfrak{g}$ if $\mathfrak{a}$ normalizes $\mathfrak{h} .$ We say that $\mathfrak{h}$ is a proper regular subalgebra of $\mathfrak{g}$ if $\textrm{rank}_{\mathbb{R}}(\mathfrak{h})<\textrm{rank}_{\mathbb{R}}(\mathfrak{g})$. Analogously to Dynkin \cite{d} we say that a real semisimple subalgebra $\mathfrak{h}\subset\mathfrak{g}$ is a (proper) real R-regular subalgebra, if it is contained in a (proper) regular subalgebra (notice that all semisimple Lie subalgebras in a simple complex Lie algebra also fall into two classes:  R-regular and S-subalgebras, according to Dynkin). Theorem \ref{thm:root-decomp} fully settles the case of real R-regular subalgebras.
\begin{theorem}\label{thm:r-sub} If $\mathfrak{h}$ is a proper real R-regular subalgebra in $\mathfrak{g}$, then no $G/H$ admits a compact Clifford-Klein form.
\end{theorem}
The following is also an easy corollary to Theorem \ref{thm:root-decomp}.
\begin{theorem}\label{thm:split} If $\mathfrak{g}$ is a simple split Lie algebra, then $(\mathfrak{g}, \mathfrak{h})$ determines a standard compact Clifford-Klein form if and only if it is contained in Table \ref{tttab1}. Also all standard Clifford-Klein forms determined by triples $(\mathfrak{g},\mathfrak{h},\mathfrak{l})$ with both $\mathfrak{h}, \mathfrak{l}$ split are contained in Table \ref{tttab1}.
\end{theorem}

\section{Preliminaries}\label{sec:prelim}
In this section, we provide more details to the formulation of Theorems \ref{thm:root-decomp},\ref{thm:r-sub} and \ref{thm:split}.
\subsection{Equivalence of triples $(\mathfrak{g},\mathfrak{h},\mathfrak{l})$}
  We consider triples $(G,H,L)$ which yield Clifford-Klein forms up to an equivalence which is given by the following well-known  fact.
\begin{proposition}\label{prop:isomorphism} If a triple $(G,H,L)$ yields a standard compact Clifford-Klein form then so does $(G,g_{1}Hg_{1}^{-1},g_{2}Lg_{2}^{-1})$ for any $g_{1},g_{2}\in G.$
\end{proposition}
\noindent This enables us to choose embeddings $\mathfrak{h}\hookrightarrow\mathfrak{g}$, $\mathfrak{l}\hookrightarrow\mathfrak{g}$ up to conjugacy in $G$, and we will do this throughout without further notice.
\noindent It is sufficient to consider the condition $G=HL$ on the Lie algebra level, because of the classical result of Onishchik.
\begin{theorem}[\cite{on}, Theorem 3.1]
Let $\mathfrak{g}$ be a reductive real Lie algebra and let $\mathfrak{g}' ,$ $\mathfrak{g}''$ be two reductive subalgebras. Let $G$ be a connected Lie subgroup with Lie algebra $\mathfrak{g}$ and let $G'$ and $G''$ be its connected Lie subgroups corresponding to Lie algebras $\mathfrak{g}'$ and $\mathfrak{g}''$. Then $G=G'G''$ if and only if $\mathfrak{g}=\mathfrak{g}'+\mathfrak{g}''.$
\end{theorem}
\begin{theorem}[\cite{bela}, Corollary 3, see also \cite{kob-d}]
If $G/H$ admits a compact Clifford-Klein form then the center of H is compact.
\end{theorem}  
Denote by $K,K_{h},K_{l}$ the maximal compact subgroups of $G,H,L,$ respectively. Let 
$$d(G):=\dim\,G/K,\,d(H):=\dim\,H/K_{h},\,  d(L):=\dim\,L/K_{l}.$$
Clearly, we can write $d(\mathfrak{g}), d(\mathfrak{h}),d(\mathfrak{l}).$
\begin{theorem}[\cite{kob}, Theorem 4.7]
Assume that $L$ acts properly on $G/H.$ The triple $(G,H,L)$ induces a standard compact Clifford-Klein form if and only if 
\begin{equation}
d(G)=d(H)+d(L),\,\text{or, equivalently}\,\,d(\mathfrak{g})=d(\mathfrak{h})+d(\mathfrak{l}).
\label{eq2}
\end{equation}
\label{koba}
\end{theorem}
Therefore the centers of $H$ and $L$ are compact. Let $\mathfrak{h}'\subset \mathfrak{h}$ be a semisimple Lie algebra given by the sum of all  simple ideals of $\mathfrak{h}$ which are of  non-compact type. Let $H'\subset H$ be the closed connected subgroup corresponding to $\mathfrak{h}' . $ We see that $L$ acts properly on $G/H'$ and that $d(G)=d(H')+d(L).$ This way the triple $(G,H',L)$ induces a standard compact Clifford-Klein form.

\subsection{Real root systems and standard Clifford-Klein forms}\label{subsec:n=n1+n2}
Throughout we work with root systems of complex semisimple Lie algebras and with real root systems of their real forms. Also, we use some relations between them, which we reproduce here in order to fix notation (see \cite{ov}, Section 4.1).
 
We may assume that there exists a Cartan involution $\theta$ of $\mathfrak{g}$ such that $\theta (\mathfrak{h}) = \mathfrak{h},$ $\theta (\mathfrak{l})=\mathfrak{l}.$ Hence the restriction of $\theta$ is a Cartan involution of $\mathfrak{h}$ and $\mathfrak{l} .$  We obtain the following Cartan decompositions
$$\mathfrak{g}=\mathfrak{k}+\mathfrak{p}, \  \mathfrak{h}=\mathfrak{k}_{h}+\mathfrak{p}_{h}, \  \mathfrak{l}=\mathfrak{k}_{l}+\mathfrak{p}_{l}, \ \ \mathfrak{k}_{h},\mathfrak{k}_{l}\subset \mathfrak{k}, \ \mathfrak{p}_{h}, \mathfrak{p}_{l}\subset \mathfrak{p}.$$
Denote by $\mathfrak{g}^c, \mathfrak{h}^c, \mathfrak{l}^c$ the complexifications of $\mathfrak{g},\mathfrak{h},\mathfrak{l},$ respectively. Let $\mathfrak{a}\subset \mathfrak{p}$ be a maximal abelian subspace of $\mathfrak{p}$ and denote by $\mathfrak{m}_{0}$ the centralizer of $\mathfrak{a}$ in $\mathfrak{k}.$ Choose a Cartan subalgebra $\mathfrak{t}$ of $\mathfrak{m}_{0}.$ Then $\mathfrak{j}:=\mathfrak{t}+\mathfrak{a}$ is a Cartan subalgebra of $\mathfrak{g}$ and the complexification $\mathfrak{j}^c$ of $\mathfrak{j}$ is a Cartan subalgebra of $\mathfrak{g}^{c}.$  Let  $\Delta$ be a root system of $\mathfrak{g}^c$ with respect to $\mathfrak{j}^c.$  Consider the root decomposition 
$$\mathfrak{g}^c=\mathfrak{j}^c+\sum_{\alpha_{c}\in \Delta} \mathfrak{g}_{\alpha_{c}}.$$
The Lie algebra $\mathfrak{j}:=i\mathfrak{t}+\mathfrak{a}$ is a real form of $\mathfrak{j}^c$ with
$$\mathfrak{j}= \{ A\in \mathfrak{j}^c \ | \ \alpha_{c}(A)\in \mathbb{R} \ \textrm{for any} \ \alpha_{c}\in \Delta  \} .$$
 The weights of the adjoint representation of $\mathfrak{a}$ constitute a root system $\Sigma\subset\mathfrak{a}^*$ (not necessarily reduced). There is a natural map $pr:\Delta\rightarrow\Sigma$ given by $\alpha_c\rightarrow \alpha_c|_{\mathfrak{a}}$. It is easy to see that $pr(\Delta)=\Sigma \cup \{ 0 \} .$ Therefore, we will call $pr(\alpha_c)=\alpha_c|_{\mathfrak{a}}\in\mathfrak{a}^*$ (if non-zero) the restricted roots. Thus,
$\Sigma = \{\alpha := \alpha_{c}|_{\mathfrak{a}} \ | \ \alpha_{c}\in\Delta   \}\setminus\{ 0\} , $
$$\mathfrak{g}_{\alpha} := \sum_{\substack{\alpha_{c}\in \Delta \\ \alpha_{c}|_{\mathfrak{a}} = \alpha}} \mathfrak{g}_{\alpha_{c}} \cap \  \mathfrak{g} .  $$
The reader should keep in mind, that we use the following convention. Roots of complex Lie algebra $\mathfrak{g}^c$ are denoted by $\alpha_c,\beta_c,,...$, the corresponding restrictions are $\alpha,\beta,...$ and so on. Vectors from the root spaces $\mathfrak{g}_{\alpha_c}$ of $\mathfrak{g}^c$ and $\mathfrak{g}_{\alpha}$ from $\mathfrak{g}$ are always denoted by $x_{\alpha}, x_{\beta},...$. The meaning of these will be always clear from the context. However, to avoid clumsy notation, in several cases we write $\alpha_i$ for complex roots. In  each case we state explicitly, which root (complex or its restriction) is used.

We get the following decomposition of $\mathfrak{g}$
$$\mathfrak{g}=\mathfrak{m}_{0}+\mathfrak{a}+\sum_{\alpha \in \Sigma} \mathfrak{g}_{\alpha} .$$
\noindent Let $B$ be the Killing form of $\mathfrak{g}.$ It extends onto $\mathfrak{g}^c$  as a Killing form of $\mathfrak{g}^c$ and we denote it by the same letter.  Consider the compact real form  $\mathfrak{g}_{u}:= \mathfrak{k}+ i\mathfrak{p},$  of $\mathfrak{g}^c ,$ and denote by $\tau$ the complex conjugation in $\mathfrak{g}^c$ with respect to $\mathfrak{g}_{u} .$  Also we denote by $\sigma$ the complex conjugation of $\mathfrak{g}^c$ with respect to $\mathfrak{g}$. It is well known (\cite{ov1}, p. 228) that there is a Hermitian non-degenerate form on $\mathfrak{g}^c$ given by the formula
$B_{\tau}(X,Y)=-B(X,\tau(Y)).$
Choose a set of positive restricted roots $\Sigma^{+}\subset \Sigma.$  There is a system of positive roots $\Delta^{+}\subset \Delta$ such that
$$\Sigma^{+}=\{  \alpha_{c}|_{\mathfrak{a}} \ | \ \alpha_{c}\in \Delta^{+} \} - \{ 0 \} .$$
The subalgebra $\mathfrak{n}:=\sum_{\alpha\in \Sigma^{+}}\mathfrak{g}_{\alpha}$ is a nilpotent subalgebra of $\mathfrak{g}.$ The decomposition
$$\mathfrak{g}=\mathfrak{k}+\mathfrak{a}+\mathfrak{n}$$
is  the Iwasawa decomposition of $\mathfrak{g} .$ Notice that $\mathfrak{a}+ \mathfrak{n} $ is a solvable subalgebra and $\mathfrak{a}$ normalizes $\mathfrak{n} .$ On the Lie group level we have
$G=KAN,$
where $K$ is a maximal compact subgroup of $G$ and $N$ is a maximal unipotent subgroup of $G.$ Analogously we obtain the Iwasawa decompositions for $H$ and $L$
$$\mathfrak{h}=\mathfrak{k}_{h}+\mathfrak{a}_{h}+\mathfrak{n}_{h}, \ \ H=K_{h}A_{h}N_{h},$$
$$\mathfrak{l}=\mathfrak{k}_{l}+\mathfrak{a}_{l}+\mathfrak{n}_{l}, \ \ L=K_{l}A_{l}N_{l}.$$ 
In the same way we get a nilpotent Lie subalgebra of $\mathfrak{g}$ given by the negative restricted roots
$$\mathfrak{n}^{-}= \sum_{\alpha\in\Sigma^{-}}\mathfrak{g}_{\alpha},$$
as well as $\mathfrak{n}_{h}^{-},\mathfrak{n}_{l}^{-}$ for $\mathfrak{h}$ and $\mathfrak{l} .$  Since  $\theta(\mathfrak{g}_{\alpha})=\mathfrak{g}_{-\alpha}$ for all $\alpha\in\Sigma$ (\cite{knapp}, Proposition 5.9) we get $\theta (\mathfrak{n})=\mathfrak{n}^{-}$.
\begin{lemma}\label{lemma:nh+nl}
Under the assumptions of Theorem \ref{thm:root-decomp}, we may assume that
$$\mathfrak{n}=\mathfrak{n}_{h}\oplus \mathfrak{n}_{l} , \ \ \mathfrak{n}^{-}=\mathfrak{n}_{h}^{-}\oplus \mathfrak{n}_{l}^{-}, \ \ \mathfrak{a}=\mathfrak{a}_{h}\oplus \mathfrak{a}_{l},$$
$$\theta (\mathfrak{h})=\mathfrak{h}, \ \ \theta (\mathfrak{l})=\mathfrak{l} .$$
\label{dobro}
\end{lemma}
\begin{proof}
By assumption  $\theta$ preserves $\mathfrak{h}$ and $\mathfrak{l} . $ It is straightforward that   $Ad_{k}(\mathfrak{k})=\mathfrak{k}$ and $Ad_{k}(\mathfrak{p})=\mathfrak{p}$ for any $k\in K$. Therefore,  
$$\theta (Ad_{k}(\mathfrak{h}))=Ad_{k} (\mathfrak{h}),\,\theta (Ad_{k}(\mathfrak{l}))=Ad_{k}(\mathfrak{l}) .$$ 
If $R$ is a connected solvable subgroup of $G$ consisting of matrices whose all eigenvalues are real, then $R$ is conjugate to a subgroup of $AN$ (\cite{abis} Proposition 3).
It follows that $A_{h}N_{h} \subset gANg^{-1}$ for some $g\in G.$ Let $g=kan$ be the Iwasawa decomposition of $g.$ Since $AN$ is a subgroup of $G$ we get
$$A_{h}N_{h}\subset kanANn^{-1}a^{-1}k^{-1}=kANk^{-1},$$
and so $k^{-1}A_{h}N_{h}k, \tilde{k}A_{l}N_{l}\tilde{k}^{-1}\subset AN,$ for some $k,\tilde{k}\in K.$ Therefore after conjugating $\mathfrak{h},$ $\mathfrak{l}$ by elements of $K$  we obtain 
$$\mathfrak{a}_{h}+\mathfrak{n}_{h}\subset\mathfrak{a}+\mathfrak{n}, \ \ \mathfrak{a}_{l}+\mathfrak{n}_{l}\subset \mathfrak{n} .$$
Since the action of $L$ on $G/H$ is proper, $A_{l}N_{l}\cap A_{h}N_{h}$ is compact. By definition,  $d(G)=\dim (AN)$ which implies
$$\mathfrak{a}+\mathfrak{n}=\mathfrak{a}_{h}\oplus\mathfrak{n}_{h}\oplus \mathfrak{a}_{l}\oplus \mathfrak{n}_{l}.$$
Therefore $\mathfrak{n}_{h},\mathfrak{n}_{l}\subset \mathfrak{n}$ and
$$\theta (\mathfrak{a}_{h}+\mathfrak{n}_{h})\subset \mathfrak{a}+\mathfrak{n}^{-}.$$
Since $\mathfrak{a}_{h},\mathfrak{a}_{l}\subset \mathfrak{a}+\mathfrak{n}\cap \mathfrak{a}+\mathfrak{n}^{-}$, we get $\mathfrak{a}_{h},\mathfrak{a}_{l}\subset \mathfrak{a}.$
\end{proof}
\begin{definition}
{\rm Embeddings $\mathfrak{h}\hookrightarrow\mathfrak{g}, \mathfrak{l} \hookrightarrow \mathfrak{g}$ fulfilling the conditions of Lemma \ref{dobro} will be called  good, and the corresponding subalgebras will be called well embedded.}
\end{definition}
\noindent In this terminology we may formulate the problem of classifying standard compact Clifford-Klein forms as a problem of describing Lie algebra triples given by good embeddings. We will say that  triples of Lie algebras $(\mathfrak{g},\mathfrak{h},\mathfrak{l})$ and $(\mathfrak{g},\mathfrak{h}',\mathfrak{l}')$ are isomorphic, if there exist inner automorphisms $\varphi_1,\varphi_2\in \operatorname{Int}(\mathfrak{g})$ such that $\mathfrak{h}'=\varphi_1(\mathfrak{h})$ and $\mathfrak{l}'=\varphi_2(\mathfrak{l})$.
\begin{proposition}\label{prop:good}
\noindent Let $(\mathfrak{g},\mathfrak{h},\mathfrak{l})$ be a triple of Lie algebras which determines a standard compact semisimple Clifford-Klein form. Then there is an isomorphic  Lie algebra triple $(\mathfrak{g},\mathfrak{h}',\mathfrak{l}')$ such that $\mathfrak{h}'$ and $\mathfrak{l}'$ are well embedded.  
\end{proposition}
\begin{proof} This is  a consequence of Lemma \ref{lemma:nh+nl} and Proposition \ref{prop:isomorphism}.
\end{proof}
\noindent Let $M_{0}\subset K$ be a closed connected subgroup corresponding to a reductive subalgebra $\mathfrak{m}_{0}$ (that is, $M_{0}$ is a maximal compact subgroup of the minimal parabolic subgroup $M_{0}AN$ of $G$).
\begin{lemma}
If $\mathfrak{h}, \mathfrak{l} \hookrightarrow \mathfrak{g}$ are good embeddings then for any $m\in M_0$, the embeddings  $Ad_{m}(\mathfrak{h}), Ad_{m}(\mathfrak{l}) \hookrightarrow \mathfrak{g}$ are good. 
\label{ge}
\end{lemma}
\begin{proof}
It easily follows from the fact that $M_{0}$ normalizes $\mathfrak{n}, $ $\mathfrak{n}^{-}$ and that $L$ acts properly on $G/H$.
\end{proof}
\begin{lemma}\label{lemma:a+n} Let $(\mathfrak{g},\mathfrak{h},\mathfrak{l})$ be a semisimple triple such that $\mathfrak{h}$ and $\mathfrak{l}$ are well embedded subalgebras. If, moreover $d(\mathfrak{g})=d(\mathfrak{h})+d(\mathfrak{l})$,  then
 $$\mathfrak{a}+\sum_{\alpha\in\Sigma} \mathfrak{g}_{\alpha} \subset \mathfrak{h}+\mathfrak{l} .$$
\end{lemma}
\begin{proof}
By Lemma \ref{lemma:nh+nl},  $\sum_{\alpha\in\Sigma}\mathfrak{g}_{\alpha}=\mathfrak{n}+\mathfrak{n}^{-}\subset \mathfrak{h}+\mathfrak{l}.$ Since $\mathfrak{a}_{h}\oplus \mathfrak{a}_{l} = \mathfrak{a} ,$ the proof follows.
\end{proof}

\subsection{Chevalley bases compatible with non-compact real forms}
Given a root $\alpha_{c}$ and a root space $\mathfrak{g}_{\alpha_{c}}$ we denote by $t_{\alpha}\in \mathfrak{j}^{c}$ the  vector determined by the equality $B(t_{\alpha},h)=\alpha_{c}(h)$ for all $h\in\mathfrak{j}^c$. Set $h_{\alpha}=2t_{\alpha}/B(t_{\alpha},t_{\alpha}).$
\begin{definition}  {\rm A Chevalley basis of $\mathfrak{g}^c$ with respect to $\mathfrak{j}^c$ is a basis of $\mathfrak{g}^c$  consisting of $x_{\alpha}\in\mathfrak{g}_{\alpha_{c}}$ and $h_{\alpha}$ with the following properties:}
\begin{enumerate}
\item $[x_{\alpha}, x_{-\alpha}]=-h_{\alpha}, \,\forall \alpha_{c}\in\Delta$,
\item {\rm for each pair $\alpha_{c},\beta_{c}\in\Delta$ such that $\alpha_{c}+\beta_{c}\in\Delta$ the constants $c_{\alpha,\beta}\in\mathbb{C}$ determined by $[x_{\alpha},x_{\beta}]=c_{\alpha,\beta}x_{\alpha+\beta}$ satisfy $c_{\alpha,\beta}=c_{-\alpha,-\beta}$.}
\end{enumerate}
\end{definition}
\noindent Throughout, we work with a special choice of the Chevalley basis compatible with a non-compact real form $\mathfrak{g}$ of $\mathfrak{g}^c$.

\begin{lemma} For the chosen $\mathfrak{j}^c$ in $\mathfrak{g}^c$ there is a base of $\sum_{\alpha_c\in\Delta}\mathfrak{g}_{\alpha_c}$ of the form
 $$\{ \tilde{x}_{\alpha} \ | \ \tilde{x}_{\alpha}\in \mathfrak{g}_{\alpha_{c}}, \ \alpha_{c}\in \Delta  \}, $$ 
such that for every $\alpha_{c}\in\Delta$
	      $$\tau (\tilde{x}_{\alpha}) = \tilde{x}_{-\alpha}, \  \ B(\tilde{x}_{\alpha}, \tilde{x}_{-\alpha})=-1.$$
\label{lemtau}
\end{lemma}
\begin{proof}
Note that $\tau (\mathfrak{j}^{c})=\mathfrak{j}^{c}$ and therefore $\tau$ permutes the root spaces. Since $x_{\alpha}+\tau(x_{\alpha})\in\mathfrak{g}_u$  we get $B(x_{\alpha}+\tau(x_{\alpha}),x_{\alpha}+\tau(x_{\alpha}))<0$. Note that $B(x_{\alpha},x_{\beta})\not=0$ if and only if $\alpha_c=-\beta_c$. Hence $B(x_{\alpha},\tau(x_{\alpha}))<0$ and therefore $\tau(x_{\alpha})=ax_{-\alpha}$ for some $a>0$. Since $\tau$ is an involution, $\tau(x_{-\alpha})=\frac{1}{a}x_{\alpha}$. Therefore, one can take 
$\tilde x_{\alpha}=\frac{1}{\sqrt{ab}}x_{\alpha}$ and $\tilde x_{-\alpha}=\frac{\sqrt{a}}{\sqrt{b}}x_{-\alpha}$. 
\end{proof}
\begin{remark}{\rm  Note that in this subsection we consider only root systems of  semisimple complex Lie algebras, and, therefore, follow the standard notation. However, in the subsequent sections  we  distinguish between real and complex roots.}
\end{remark}

\subsection{Some properties of  root systems}
In this subsection we establish some properties of  abstract root systems. Our aim is Lemma \ref{lemma:root-sum} which describes subsets of the root system consisting of roots simultaneously vanishing on a pair of vectors. This may be known, but we did not find any appropriate reference. Let $\Delta\subset \mathbb{R}^{n}$ be an indecomposable root system in the Euclidean space $(\mathbb{R}^{n},(,))$  (\cite{ov}, Chapter 3). The system $\Delta$ is not assumed to be reduced. Let $\Delta^{+}\subset \Delta$ be a subset of positive roots and  $\Pi\subset \Delta^{+}$  the subset of simple roots. There is a unique maximal root $\beta \in \Delta^{+}$ such that for every $\alpha\in \Delta^{+}$  vector $\beta - \alpha$ is a combination of simple roots with non-negative coefficients.
\begin{lemma}\label{lemma:max-root}
Let $A$ be a vector in the interior of the positive Weyl chamber. Then $\beta (A) > \alpha (A)$ for any $\alpha\in \Delta^{+}- \{ \beta \} .$
\end{lemma} 
\begin{proof}
Since $A$ is in the interior of the positive Weyl chamber, $\alpha_{i} (A) > 0$ for any simple root $\alpha_{i}\in \Pi .$ Thus for $\alpha\in\Delta^{+},$ $\alpha \neq \beta$ we obtain
$(\beta - \alpha )(A) > 0.$
\end{proof}

\begin{proposition}[\cite{ov}, Chapter 3, Proposition 1.2]
Let $\gamma_{1}, \gamma_{2}\in \Delta .$ If $(\gamma_{1},\gamma_{2})<0$ then $\gamma_{1}+\gamma_{2}\in \Delta .$
\label{roots}
\end{proposition}
\begin{lemma}
If $\beta \in \Delta^{+}$ is such that for any $\alpha\in\Delta^{+}$ the vector $\beta + \alpha$ is not a root then $\beta$ is the highest root.
\end{lemma}
\begin{proof}
Assume that $\beta, \gamma\in\Delta^{+}$ are distinct and have the  property that $\beta+\alpha$ and $\gamma+\alpha$ are not roots for any $\alpha\in\Delta^{+}.$ By Proposition \ref{roots} for any simple root $\alpha_{i}$ we get $(\beta , \alpha_{i}), (\gamma , \alpha_{i}) \geq 0.$ As $\beta \neq \gamma ,$ for some simple root $\alpha_{i}$ we get $(\beta , \alpha_{i}) \neq (\gamma , \alpha_{i})$.
So $(\beta + \gamma , \alpha_{i})>0.$ By Proposition \ref{roots} (setting $\gamma_{1}=\alpha + \beta ,$ $\gamma_{2}=-\alpha_{i}$) we see that 
$\theta := \beta + \gamma -\alpha_{i}$
is a root. Since $\beta$ is the highest root, $\gamma$ is a positive root and $\alpha_{i}$ is simple (a lowest root) therefore $\gamma - \alpha_{i} = 0$ (otherwise $\beta - \theta$ written as a combination of positive roots would have some negative coefficients). But  $\gamma$ is a simple root and so does not have the assumed  property that $\gamma+\alpha$ is not a root for every positive $\alpha$ (unless $\Delta = \{ \alpha , -\alpha  \}  $ in which case $\gamma = \beta$).  
\end{proof}
\begin{lemma}\label{lemma:root-sum}
Let $X,H\in\mathbb{R}^{n}$ be  non-zero. Define 
$$C_{X}:= \{ \alpha\in\Delta \ | \ \alpha (X) = 0   \} , \ \ C_{H}:= \{ \alpha\in\Delta \ | \ \alpha (H) = 0   \} .$$
 Then $\Delta \neq C_{X} \cup C_{H}.$ 
\end{lemma}
\begin{proof}
Assume that $\Delta = C_{X} \cup C_{H} . $ Consider sets $C_{H}\setminus C_{X}$ and $C_{X}\setminus C_{H}.$ If one of them is empty then all the roots in $\Delta$ are orthogonal to $X$ or to $H$, but this is impossible. Therefore there exist simple roots $\gamma_{1}, \gamma_{2}$ (for any choice of a subset of simple roots) such that
$$\gamma_{1}\in C_{H}\setminus C_{X}, \ \ \gamma_{2}\in C_{X}\setminus C_{H},$$
because any root can be represented as a linear combination of simple roots. Since $C_{H}\cap C_{X}$ is a root system, we can take the set of positive roots $C_{HX}^{+} \subset C_{H}\cap C_{X}.$ Let 
$$C_{X}^{+} := \{ \alpha\in C_{H}\setminus C_{X} \ | \ \alpha (X) >0  \} ,$$
$$C_{H}^{+} := \{ \alpha\in C_{X}\setminus C_{H} \ | \ \alpha (H) >0  \} ,$$
One can see that for $\alpha\in C_{X}^{+}$ and $\beta\in C_{H}^{+}$, vector $\alpha + \beta$ is not a root, because otherwise one would obtain a root which is non-zero on $H$ and on $X.$ Consider the set
$$\tilde{\Delta}^{+} := C_{HX}^{+}\cup C_{X}^{+}\cup C_{H}^{+}. $$
Note that for any $\alpha\in\Delta$ exactly one of the roots $\alpha ,$ $-\alpha$ is contained in $\tilde{\Delta}^{+} .$ Also if $\alpha , \gamma \in \tilde{\Delta}^{+}$ and $\alpha + \gamma $ is a root, then $\alpha + \gamma \in \tilde{\Delta}^{+}.$ Therefore $\tilde{\Delta}^{+}$ is the set of positive roots from $\Delta $ with the  property
$$\forall_{\alpha\in\tilde{\Delta}^{+}}\alpha (X)\geq 0, \ \alpha (H)\geq 0.$$
Consider the highest root $\tilde{\beta}\in \tilde{\Delta}^{+}.$ Since $\Delta$ is indecomposable  $\tilde{\beta}$ is a linear combination of all simple roots of $\tilde{\Delta}^{+}$ with  positive coefficients. But  we have seen before that there are simple roots $\tilde{\gamma}_{1} ,$ $\tilde{\gamma}_{2}\in \tilde{\Delta}^{+}$ such that
$$\tilde{\gamma}_{1} (X)>0, \ \ \tilde{\gamma}_{2}(H)>0.$$
Thus $\tilde{\beta} (X)>0$ and $\tilde{\beta} (H)>0.$ A contradiction.
\end{proof}

\section{Proof of Theorem \ref{thm:root-decomp}}
By Propositions \ref{prop:iwasawa} and  \ref{prop:good} we can restrict ourselves to good embeddings $\mathfrak{h}\hookrightarrow\mathfrak{g}$ and $\mathfrak{l}\hookrightarrow\mathfrak{g}$.
\begin{proposition} Under the assumptions of Theorem \ref{thm:root-decomp},
if there does not exist $0\neq X\in \mathfrak{t}$ such that $B(X,U)=0$ for any $U\in \mathfrak{h}+\mathfrak{l} ,$ then  $\mathfrak{g}=\mathfrak{h}+\mathfrak{l}$. 
\label{propx}
\end{proposition}
\begin{proof} Let $0 \neq X\in \mathfrak{g},$ $X\notin \mathfrak{h}+\mathfrak{l} .$ By Lemma \ref{lemma:a+n},  $\mathfrak{p} \subset \mathfrak{h}+\mathfrak{l}$, therefore, we may assume that $X\in \mathfrak{k} $. The Killing form $B$ is negative-definite on $\mathfrak{k}$, so  we can assume that  $X$ is $B$-orthogonal to $\mathfrak{h}+\mathfrak{l} .$ For any $\alpha \in \Sigma$ and any non-zero $x_{\alpha}\in \mathfrak{g}_{\alpha}$ there exists $H\in \mathfrak{a}$ such that $[H,x_{\alpha}]=x_{\alpha}$, therefore 
$B(Y,x_{\alpha})=B(Y,[H,x_{\alpha}])=B([Y,H],x_{\alpha})=0 ,$ for any $Y\in\mathfrak{m}_0$.
We see that  $\mathfrak{m}_{0}$ is $B$-orthogonal to $\mathfrak{a}+\sum_{\alpha\in\Sigma}\mathfrak{g}_{\alpha}$ and so ${(\mathfrak{h}+\mathfrak{l})}^{\perp}\subset \mathfrak{m}_{0}.$ Since $M_{0}$ is compact, any vector $X\in \mathfrak{m}_{0}$ is conjugate by some element of $M_{0}$ to a vector in the Cartan subalgebra $\mathfrak{t}$ of $\mathfrak{m}_{0}$ and so by Lemma \ref{ge} we may assume that $X\in \mathfrak{t} .$
\end{proof}

\begin{lemma} Under the assumptions of Theorem \ref{thm:root-decomp},
$$\mathfrak{n}^c=\sum_{\alpha_{c}\in \Delta_{m}^{+}} \mathfrak{g}_{\alpha_{c}}, \ \ (\mathfrak{n}^{-})^c=\sum_{\beta_c\in\Delta_m^{-}}\mathfrak{g}_{\beta_c}.$$

\label{le1}
\end{lemma}
\begin{proof}
Because of symmetry, we write the necessary formulas only for   $\mathfrak{n}$ (although $\mathfrak{n}^{-}$ is also used). By the formulas from Preliminaries, for any $\alpha \in \Sigma^{+}$ 
$$\mathfrak{g}_{\alpha} := \sum_{\substack{\alpha_{c}\in \Delta_{m}^{+} \\ \alpha_{c}|_{\mathfrak{a}=\alpha}}} \mathfrak{g}_{\alpha_{c}}\cap \mathfrak{g}.$$
It follows that 
$\mathfrak{n}\subset \sum_{\alpha_c\in\Delta_m^+}\mathfrak{g}_{\alpha_c}.$
We also know that 
$\mathfrak{m}_{0}^c=\sum_{\alpha_{c}\in \Delta_{m}}\mathfrak{g}_{\alpha_c}.$
Since $\Delta=\Delta_{m}^{+}\cup\Delta_{m}\cup\Delta_{m}^{-}$
 and 
$$\mathfrak{g}^c=\mathfrak{m}_{0}^c+\mathfrak{a}^c+{(\mathfrak{n}^{-})}^c+\mathfrak{n}^c$$
we get $\dim_{\mathbb{R}}\mathfrak{n}=\dim_{\mathbb{C}}\mathfrak{n}^c$ and the proof follows.
\end{proof}
\begin{lemma} Under the assumptions of Theorem \ref{thm:root-decomp},
\begin{itemize}
\item The involution $\sigma$ permutes the root spaces given by $\Delta^{+}_{m}$ so that if $\alpha_{c}\in \Delta_{p}$ then $\sigma (\mathfrak{g}_{\alpha_{c}})\subset \mathfrak{g}_{\alpha_{c}^{'}}$ for some $\alpha_{c}^{'}\in \Delta_{n} $ (by abuse of notation we write $\sigma (\Delta_{p})=\Delta_{n}$). Thus
\begin{equation}
\Delta_{p}= \{ \alpha_{1},...,\alpha_{k} \} , \ \Delta_{0}=\{ \beta_{1},...,\beta_{t})  \} , 
\Delta_{n} = \{ \alpha_{k+1},...,\alpha_{2k}   \},\label{eq4}
\end{equation}
where $\alpha_i,\beta_j,$ denote the complex roots of $\mathfrak{g}^{c}.$
\item $\Delta_{p}\neq \emptyset .$ 
\end{itemize}
Analogous properties hold for $\Delta_{p}^{-}, \Delta_{n}^{-}$
\label{sigma}
\end{lemma}
\begin{proof}
Note that  $\sigma (\mathfrak{j}^{c})=\mathfrak{j}^{c}$, therefore,  $\sigma$ permutes the root spaces of $\mathfrak{g}^{c}.$ Let us  show that $\sigma(\Delta_{p})= \Delta_{n}$.  This follows from the observation:  $\sigma (iX) = -iX$ and thus for any $y\in \mathfrak{g}_{\alpha_{c}}$ 
$$[iX,\sigma (y)]=-\sigma[iX,y] =-\sigma (\alpha_{c} (iX)y)=-\alpha_{c}(iX)\sigma (y) .$$ 
By Lemma \ref{lemma:root-sum}, $\Delta_p\not=\emptyset$. Indeed if $\Delta_{p}=\emptyset$ then also $\Delta_{n}=\emptyset$ and (since $\tau (\Delta_{p})=\Delta_{p}^{-}$)  $\Delta_{p}^{-} = \Delta_{n}^{-}=\emptyset .$ But then $\Delta = \Delta_{m}\cup \Delta_{m}^{-}\cup \Delta_{0},$ and for any $0\neq H\in \mathfrak{a}$
$$\Delta_{m}, \Delta_{m}^{-}\subset \{  \alpha_{c}\in\Delta \ | \ \alpha_{c}(H)=0 \} ,$$
$$\Delta_{0}\subset \{  \alpha_{c}\in\Delta \ | \ \alpha_{c}(iX)=0 \} .$$
\end{proof}
\noindent In accordance with Lemma \ref{sigma}, 
$$\mathfrak{n}^c=\sum_{\alpha_c\in\Delta_p}\mathfrak{g}_{\alpha_c}\oplus\sum_{\beta_c\in\Delta_0}\mathfrak{g}_{\beta_c}\oplus\sum_{\alpha_c\in\Delta_n}\mathfrak{g}_{\alpha_c},$$
where $\Delta_{p}, $ $\Delta_{0}$ and $\Delta_{n}$ are given by (\ref{eq4}).
Let 
$$Z=\sum_{\alpha_c\in\Delta_p\cup\Delta_n}\mathfrak{g}_{\alpha_c}\subset\mathfrak{n}^c,$$
and let $\pi: \mathfrak{n}^{c}\rightarrow Z$ be the projection onto $Z$ given in the base described in Lemma \ref{lemtau}.
\begin{lemma} Assume that $X\in\mathfrak{t}$ satisfies the assumptions of Theorem \ref{thm:root-decomp}.
Let  $\omega: Z\times Z\rightarrow\mathbb{C}$ be a map defined  by the formula 
$$\omega (S_{1},S_{2})=-B_{\tau}([iX,S_{1}],S_{2}), \ \textrm{for any} \ S_{1},S_{2}\in Z.$$
The map $\omega$ is a non-degenerate Hermitian 
 form.
\label{le3}
\end{lemma}
\begin{proof} The proof follows from the fact that $B_{\tau}$ is a Hermitian  form. Thus, only the non-degeneracy should be checked. Choose any $S\in Z$.  We need to find $S'$ such that $\omega(S,S')\not=0$. Take $S'=[iX,S]$.
One  checks that
$$\omega (S,[iX,S])=-B_{\tau}([iX,S],[iX,S])>0,$$
for any $S\in Z$. This follows from a straightforward calculation, if one writes $S=\sum_{\alpha_c\in\Delta_p\cup\Delta_n}c_{\alpha}x_{\alpha},c_{\alpha}\in\mathbb{C}$ and uses the formula $\tau(x_{\alpha})=x_{-\alpha}$  established in Lemma \ref{lemtau}, together with $iX\in\mathfrak{t}\subset\mathfrak{j}^c$.
\end{proof}

\begin{remark}
{\rm Note that $\tau$ is a complex conjugation of $\mathfrak{g}^c$ with respect to the compact real form $\mathfrak{u}=\mathfrak{k}+i\mathfrak{p}$. By the assumptions that the Cartan decomposition $\mathfrak{g}=\mathfrak{k}+\mathfrak{p}$ is inherited by $\mathfrak{h}$ and $\mathfrak{l}$ the compact real forms of $\mathfrak{h}^c$ and $\mathfrak{l}^c$ satisfy $\mathfrak{u}_h\subset\mathfrak{u}$ and $\mathfrak{u}_l\subset\mathfrak{u}$, and $\tau$ restricts onto $\mathfrak{h}^c$ and $\mathfrak{l}^c$ as a complex conjugation with respect to $\mathfrak{u}_h$ and $\mathfrak{u}_l$. Thus since $0\neq X\in\mathfrak{t}$ satisfies $X\perp \mathfrak{h}+\mathfrak{l}$, $\omega$  vanishes  on $\mathfrak{n}_{h}^c$ and on $\mathfrak{n}_{l}^c .$}
\end{remark}
\begin{proposition}\label{prop:zh} Assume that $X\in\mathfrak{t}$ satisfies the assumptions of Theorem \ref{thm:root-decomp} and put $Z_h=\pi(\mathfrak{n}_h^c), Z_l=\pi(\mathfrak{n}_l^c)$. Then
$$Z=Z_{h}\oplus [X,Z_{h}]=Z_{l}\oplus [X,Z_{l}]=Z_{h}\oplus Z_{l}=Z_{h}\oplus [X,Z_{l}].$$
\label{propbbb}
\end{proposition}
\begin{proof}
Observe that  the natural projection $\pi:\mathfrak{n}^c\rightarrow Z$ has the property
$\pi (\mathfrak{n}_{h}^c)+\pi (\mathfrak{n}_{l}^c)=Z$ and so $Z=Z_h+ Z_l$.
Note that under the assumptions of the Proposition  $\omega$  vanishes on $Z_h$ and on $Z_l.$ This can be shown as follows. One needs to keep in mind that $X\perp (\mathfrak{h}+\mathfrak{l})$ with respect to the Killing form $B$. Observe that $[iX,\pi(N_1)]=[iX,N_1]$, because $\alpha_c(iX)=0$ for all $\alpha_c\in\Delta_0$, and so
$$\omega(\pi_1(N_1),\pi (N_2))=\omega(N_1,N_2).$$
We see that $Z_{h} \cap [X,Z_{h}] = \{ 0 \} $, since otherwise $\omega $ would be non-zero on $Z_{h}$ by Lemma \ref{le3}. Analogously $Z_{l} \cap [X,Z_{l}]= \{ 0 \}$. Clearly, $\operatorname{ad}X(Z)=Z $ and, therefore $\operatorname{ad}X$ is an isomorphism. Thus, $\dim\,Z_h\leq \frac{1}{2}\dim\,Z$ and analogously, $\dim\,Z_l\leq \frac{1}{2}\dim Z$. Hence, the only possibility is that $Z=Z_h\oplus Z_l$ and  $\dim Z$ is even, 
$Z= Z_{h} \oplus [X,Z_{h}]=Z_l\oplus [X,Z_l]$.
\end{proof}
\subsection{Completion of proof of Theorem \ref{thm:root-decomp}}
\noindent{\it Proof of 1)}.
Since $\dim Z_{h} = \frac{1}{2}\dim Z=k$ by Proposition \ref{prop:zh}, $Z_{h}$ contains a non-zero vector of the form
$$A_{j}:=ax_{\alpha_{j}}+a_{1}^{j}x_{\alpha_{k+1}}+...+a_{k}^{j}x_{\alpha_{2k}}.$$
If $a=0$ then one easily shows that $\omega (A_{j},A_{j}) \neq 0$ and so $\omega$ is non-zero on $Z_{h}$. 
Clearly, $\{  S_{h}^{j} \ | \ j=1,...,k \} $ are linearly independent.
\vskip6pt
\hfill$\square$
 
\noindent{\it Proof of 2) and 3).}
Assume that there exists $A_{H}\in \mathfrak{a}_{h}$ and $S_{h}^{j}$ such that
$$[A_{H},S_{h}^{j}]\neq \alpha_{j}(A_{H})S_{h}^{j}.$$
If $\alpha_{j}(A_{H}) = 0$ then put $Y:=[A_{H},S_{h}^{j}]\in Z_{h}.$ Otherwise scale $A_{H}$ so that $\alpha_{j} (A_{H})=1$ and put $Y:=S_{h}^{j}-[A_{H},S_{h}^{j}]\in Z_{h}.$ One easily shows that $\omega (Y,Y) \neq 0$ and so $\omega$ is non-zero on $Z_{h}.$ The second part of the Lemma follows from the fact that $\mathfrak{n}_{h}^{c}\subset \mathfrak{n}^c .$
\vskip6pt
\hfill$\square$
\begin{remark}
{\rm Analogous  holds for $S_{l}^{i}.$}
\end{remark}

\section{Proof of Theorem \ref{thm:r-sub}}

\begin{lemma}
None of the non-zero root vectors $x_{\beta}$ for  $\beta_c\in \Delta_{0}$   centralizes $Z+\tau (Z) .$ 
\label{lemz}
\end{lemma}
\begin{proof}
Assume there exists $x_{\beta},\beta_c\in\Delta_0$ such that  $[x_{\beta},Z+\tau(Z)]=0$.  Take $H_{\beta}\in i\mathfrak{t}+\mathfrak{a}$, defined, as usual by the condition $B(\tilde{H},H_{\beta})=\beta(\tilde{H})$ for any $\tilde{H}\in i\mathfrak{t}+\mathfrak{a}$.  Denote by $\mathfrak{g}^{x_{\beta}}$ the centralizer of $x_{\beta}$ in $\mathfrak{g}^c$. By \cite{col}, Lemma 3.4.3
$$\mathfrak{g}^{x_{\beta}} = \oplus_{j\geq 0} (\mathfrak{g}^c\cap \mathfrak{g}_{j}),$$
where $\mathfrak{g}_j$ denote the eigenspaces of $\operatorname{ad}\,H_{\beta}$ (which are integers $j\geq 0$). Write $H_{\beta}=iT+H$, where $T\in\mathfrak{t},H\in\mathfrak{a}$. Note that $H\not=0$, since $\beta\in\Delta_0$. We have
$$H=\frac{1}{2}(H_{\beta}+\sigma(H_{\beta})).$$
Define
 $$ C_X=\{\alpha_c\in\Delta\,|\alpha_c(X)=0\},\,C_H=\{\alpha_c\in\Delta\,\,|\,\alpha_c(H)=0\}.$$
Consider the decomposition
$$\Delta=\Delta_{m}\cup\Delta_p\cup\Delta_n\cup\Delta_0\cup\Delta_p^{-}\cup\Delta_n^{-}.$$
 Note that for any $\mu_c\in\Delta_{m}$, $\mu_c(H)=0$. Assume that 
\begin{equation}
\Delta_p\cup\Delta_n\cup\Delta_p^{-}\cup\Delta_n^{-} \subset C_H. \label{eq5}
\end{equation}
Since $\Delta_0\subset C_X$, and all roots from $\Delta_{m}$ vanish on $H,$ equality (\ref{eq5}) would imply  
$$\Delta=C_X\cup C_H,$$
but this contradicts Lemma \ref{lemma:root-sum}. Thus there exists $\alpha_c,-\alpha_{c}\in\Delta_p\cup\Delta_n\cup\Delta_p^{-}\cup\Delta_n^{-}$ such that $\alpha_c(H)\not=0$. We have one of the following possibilities
\begin{enumerate}
	\item $\alpha_{c}(iT)\neq -\alpha_{c}(H).$ In this case (since one can choose between $\alpha_c$ and $-\alpha_c$) we may assume that $\alpha_{c}(H_{\beta})<0.$
	\item $\alpha_{c}(iT)= -\alpha_{c}(H).$ In this case we see that $\sigma (\alpha_{c})(H_{\beta})=\alpha_{c}(-iT+H)=2\alpha_{c}(H)\neq 0$ and by Lemma \ref{sigma}, $\sigma (\alpha_{c})\in \Delta_p\cup\Delta_n\cup\Delta_p^{-}\cup\Delta_n^{-}.$ In this case (since one can choose between $\sigma(\alpha_c)$ and $-\sigma(\alpha_c)$) we may assume that $\sigma(\alpha_{c})(H_{\beta})<0.$
\end{enumerate}
Therefore in any case there exists $\alpha_{c}\in\Delta_p\cup\Delta_n\cup\Delta_p^{-}\cup\Delta_n^{-}$ such that $\alpha_{c}(H_{\beta})<0.$ But for any root vector $x_{\alpha}\in Z+\tau(Z)$ we have $\operatorname{ad}_{H_{\beta}}(x_{\alpha})=\alpha_c(H_{\beta})x_{\alpha}$ and so for $x_{\alpha}\in\mathfrak{g}^{x_{\beta}},$ $j=\alpha_c(H_{\beta})$. This is a contradiction with the inequality $j\geq 0$. The proof of the Lemma is complete.
\end{proof}

\begin{proposition}
Let $0\neq A\in \mathfrak{a}$ and assume that $\mathfrak{h}$ is contained in the centralizer of $A$ in $\mathfrak{g} .$ Then $G/H$ does not admit standard compact Clifford-Klein forms.
\end{proposition}
\begin{proof}
There are two possibilities for the triple $(\mathfrak{g},\mathfrak{h},\mathfrak{l})$: either $\mathfrak{g}=\mathfrak{h}+\mathfrak{l}$ (and then it must belong to  Table \ref{tttab1}), or $\mathfrak{g}\not=\mathfrak{h}+\mathfrak{l}$. However, no proper real R-subalgebra is contained in this table. Thus, we should assume the conclusion of Theorem \ref{thm:root-decomp} holds. First, we claim that that there exists a root in $\alpha\in \Delta_{p}\cup \Delta_{n}$ such that $\alpha (A)\neq 0.$ Indeed, since $\mathfrak{g}$ is simple, then for some $\hat{\gamma} \in \Delta_{p}\cup \Delta_{0}\cup \Delta_{n},$ we have $\hat{\gamma} (A)\neq 0.$ If $\hat{\gamma}\in \Delta_{0}$ then by Lemma \ref{lemz} we have $\hat{\gamma} = \alpha_{1}+\alpha_{2}$ for some $\alpha_{1},\alpha_{2}\in \Delta_{p}\cup \Delta_{n}\cup \Delta_{p}^{-}\cup \Delta_{n}^{-}$ and so $\alpha_{1}(A)\neq 0$ or $\alpha_{2}(A)\neq 0.$ In the latter case put $\alpha=\pm\alpha_1$, or $\alpha=\pm\alpha_2$. Thus, in any case $\alpha(A)\not=0$ for some $\alpha\in\Delta_p\cup\Delta_n$.

If $\mathfrak{h}$ is a proper real R-subalgebra then $\mathfrak{h}$ is contained in the centralizer of some non-zero $A\in\mathfrak{a} .$ But now by Theorem \ref{thm:root-decomp}, part 3, $\alpha|_{\mathfrak{a}_h}=\gamma\in\Sigma_h$ (or $\sigma(\alpha)|_{\mathfrak{a}_h}=\gamma\in\Sigma_h$) and so $\mathfrak{h}_{\gamma}^c$ does not centralize $A$.  The proof of the Proposition and the proof of Theorem \ref{thm:r-sub} are complete.
\end{proof}

\section{Proof of Theorem \ref{thm:split}}
We have the following
\begin{proposition}
Under the assumptions of Theorem \ref{thm:split}, if $\mathfrak{g}$ is split or both $\mathfrak{h}, \mathfrak{l}$ are split then $\mathfrak{g}= \mathfrak{h} + \mathfrak{l}.$
\end{proposition}
\begin{proof}
If $\mathfrak{g}$ is split then $\mathfrak{m}_{0}= \{ 0 \} $ and the theorem follows from Lemma \ref{lemma:a+n}.

Assume that $\mathfrak{h}, \mathfrak{l}$ are both split. Recall that every root space of a split semisimple Lie algebra is one-dimensional. Therefore, by Theorem \ref{thm:root-decomp} we have for every $i, 1\leq i \leq k$
$$S_{h}^{i}= x_{\alpha_{i}} + a_{i}\sigma (x_{\alpha_{i}}), \ \ S_{l}^{i}= x_{\alpha_{i}} + b_{i}\sigma (x_{\alpha_{i}}),$$
for some $a_{i}, b_{i}\in \mathbb{C}- \{ 0 \}.$ Also for every $i$ we have
$$0=\omega (S_{h}^{i},S_{h}^{i})=B([iX,S_{h}^{i}],\tau (S_{h}^{i}))= \alpha_{i}(iX)(1-\lvert a_{i} \lvert),$$
thus $\lvert a_{i} \lvert=\lvert b_{i} \lvert=1.$ Notice that by  Lemma \ref{lemma:ad-root} below
$$\operatorname{Ad}_{\operatorname{exp}(tX)}S_{h}^{i}=e^{-i\alpha_{i}(itX)}x_{\alpha_{i}}+a_{i}e^{i\alpha_{i}(itX)}\sigma (x_{\alpha_{i}}).$$ 
By Lemma \ref{sigma}, $\alpha_{1}\in\Delta_{p}.$ If $t_{0}$ is such that $\alpha_{1}(it_{0}X)=\frac{1}{2}(\operatorname{Arg}(b_{1})-\operatorname{Arg}(b_{2}))$ then
$$e^{\alpha_{1}(it_{0}X)}\operatorname{Ad}_{\operatorname{exp}(t_{0}X)}S_{h}^{1}=S_{l}^{1}.$$
Since for every $t\in\mathbb{R},$ $\operatorname{Ad}_{\operatorname{exp}(tX)} Z=Z$ we have $Z_{\operatorname{Ad}_{\operatorname{exp}t_{0}X} \mathfrak{h}}+ Z_{l}\neq Z.$ But this contradicts Lemma \ref{ge}.
\end{proof}

\begin{lemma}\label{lemma:ad-root}
For every $x_{\alpha_{s}}\in \mathfrak{g}_{\alpha_{s}},$ $s=1,...,2k$ and every $t\in \mathbb{R}$ we have
$$\operatorname{Ad}_{\operatorname{exp}(tX)}(x_{\alpha_{s}})=e^{-i\alpha_{s}(itX)}x_{\alpha_{s}}.$$
\label{exp}
\end{lemma}
\begin{proof}
For every $A,B\in \mathfrak{g}_{c}$ we have
$$\operatorname{Ad}_{\operatorname{exp}(A)}(B)=\operatorname{exp}(\operatorname{ad}_{A})(B)=B+[A,B]+\frac{1}{2!}[A,[A,B]]+... \ .$$
Therefore
$$\operatorname{Ad}_{\operatorname{exp}(tX)}(x_{\alpha_{s}})=x_{\alpha_{s}}+ [tX,x_{\alpha_{s}}]+ \frac{1}{2!}[tX,[tX,x_{\alpha_{s}}]]+...=$$
$$x_{\alpha_{s}}+\alpha_{s}(tX)x_{\alpha_{s}}+ \alpha^{2}(tX)x_{\alpha_{s}}+... = e^{\alpha_{s}(tX)}x_{\alpha_{s}}=e^{-i\alpha_{s}(itX)}x_{\alpha_{s}}.$$
\end{proof}
We see that necessarily $\mathfrak{g}=\mathfrak{h}+\mathfrak{l}$, but this contradicts the Onishchik table. Theorem \ref{thm:split} is proved.

\section{Comparison with known results}
In recent years the following results concerning standard compact Clifford-Klein forms were obtained.
\begin{theorem}[\cite{tojo}]
Let $G/H$ be a non-compact irreducible simple symmetric space which admits a standard compact Clifford-Klein form. Then either $(\mathfrak{g},\mathfrak{h})$  is contained in Table \ref{tttab1} or $H$ is compact.
\end{theorem}
\begin{theorem}[\cite{bjt}]
Let $G/H$ be a non-compact reductive homogeneous space of a real linear simple exceptional Lie group $G.$ Then $G/H$ admits a standard compact Clifford-Klein form if and only if $H$ is compact.
\end{theorem}

Applying Theorems \ref{thm:r-sub} and \ref{thm:split} we obtain the following examples of homogeneous spaces without standard compact Clifford-Klein forms. 
\begin{corollary}\label{cor:split}
Let $G$ be a linear connected Lie group whose  Lie algebra is one of the following
$$\mathfrak{sl}(n,\mathbb{R}),n>1,\mathfrak{so}(n,n),n=3,5,7,  n>9,$$ 
$$\mathfrak{so}(n,n+1), n=2, n>3,$$
$$\mathfrak{sp}(n,\mathbb{R}),n>1.$$
 Choose any reductive subgroup $H\subset G$ such that $G/H$ is non-compact. Then $G/H$ admits a standard compact Clifford-Klein form if and only if $H$ is compact. 
\end{corollary}

\begin{corollary}\label{cor:su}
Let $G$ be a linear connected Lie group whose Lie algebra is $\mathfrak{su}(n,m),$ $m\geq n>2.$ Let $H\subset G$ be a reductive subgroup whose Lie algebra  is a subalgebra of a proper regular subalgebra $\mathfrak{su}(n-1,m-1)$ such that $G/H$ is non-compact. Then $G/H$ admits a standard compact Clifford-Klein form if and only if $H$ is compact.
\end{corollary}

\begin{corollary}\label{cor:so}
Let $G$ be a linear connected Lie group whose Lie algebra is $\mathfrak{so}(n,m),$ $m+1> n>8.$ Let $H\subset G$ be a reductive subgroup whose Lie algebra is  a subalgebra of a proper regular subalgebra $\mathfrak{so}(n-1,m-1)$ such that $G/H$ is non-compact. Then $G/H$ admits a standard compact Clifford-Klein form if and only if $H$ is compact.
\end{corollary}

\begin{corollary}\label{cor:sp}
Let $G$ be a linear connected Lie group with Lie algebra $\mathfrak{sp}(n,m),$ $m\geq n>1.$ Let $H\subset G$ be a reductive subgroup corresponding to a subalgebra of a proper regular subalgebra $\mathfrak{sp}(n-1,m-1)$ such that $G/H$ is non-compact. Then $G/H$ admits a standard compact Clifford-Klein form if and only if $H$ is compact.
\end{corollary}
\noindent The proof of Corollaries \ref{cor:split}-\ref{cor:sp} is a direct application of Theorems \ref{thm:r-sub} and \ref{thm:split}. Table 4 in \cite{ov} yields all split real forms of complex simple Lie algebras. The regularity of subalgebras in these corollaries can be checked directly. In fact, one can easily notice that if $\mathfrak{h}^c\subset\mathfrak{g}^c$ is a regular subalgebra, and $\mathfrak{h}\subset\mathfrak{g}$ is an embedding of the real forms compatible with the Cartm involution, then the regularity of $\mathfrak{h}^c$ in $\mathfrak{g}^c$ implies the regularity of the real embedding. The regularity of the complexified embeddings in Corollaries \ref{cor:su}-\ref{cor:sp} is clear (one can consult \cite{ov} for the root description of regular subalgebras in the complex case).
\vskip6pt
\noindent {\bf Acknowledgment.} The first named author was supported by the National Science Center (Poland), grant NCN 2018/31/D/ST1/00083. The second named author acknowledges the support of the National Science Center (Poland), grant NCN 2018/31/B/ST1/00053.

\vskip6pt
Faculty of Mathematics and Computer Science
\vskip6pt
\noindent University of Warmia and Mazury
\vskip6pt
\noindent S\l\/oneczna 54, 10-710 Olsztyn, 
Poland
\vskip6pt
e-mail adresses:
\vskip6pt
mabo@matman.uwm.edu.pl (MB)
\vskip6pt
tralle@matman.uwm.edu.pl (AT)

\end{document}